\documentclass[12pt]{amsart}
\usepackage{amssymb}
\usepackage[all]{xy}

\input xy
\xyoption{all}

\newtheorem{lemma}{Lemma}[section]
\newtheorem{corollary}[lemma]{Corollary}
\newtheorem{theorem}[lemma]{Theorem}
\newtheorem{proposition}[lemma]{Proposition}

\theoremstyle{definition}
\newtheorem{remark}[lemma]{Remark}
\newtheorem{definition}[lemma]{Definition}

\newtheorem{examples}[lemma]{Examples}

\begin{document}
\title{Decomposable Leavitt path algebras for arbitrary graphs}

\author{G. Aranda Pino}
\address{Departamento de \'{A}lgebra, Geometr\'{\i}a y Topolog\'{\i}a\\
Universidad de M\'{a}laga\\
29071, M\'{a}laga, Spain}
\email{g.aranda@uma.es}

\author{A. R. Nasr-Isfahani}
\address{Department of Mathematics\\
University of Isfahan\\
P.O. Box: 81746-73441, Isfahan, Iran\\ and School of Mathematics, Institute for Research in Fundamental Sciences (IPM), P.O. Box: 19395-5746, Tehran, Iran}
\email{a$_{-}$nasr$_{-}$isfahani@yahoo.com }

\subjclass[2000]{16W99}

\keywords{Leavitt path algebra, indecomposable algebra, decomposable algebra, arbitrary graph}

\begin{abstract} For any field $K$ and for a completely arbitrary graph $E$, we characterize the Leavitt path algebras $L_K(E)$ that are indecomposable (as a direct sum of two-sided ideals) in terms of the underlying graph. When the algebra decomposes, it actually does so as a direct sum of Leavitt path algebras for some suitable graphs. Under certain finiteness conditions, a unique indecomposable decomposition exists.
\end{abstract}

\maketitle


\section{Introduction}

Leavitt path algebras can be regarded as the algebraic counterparts of the graph $C^*$-algebras, the descendants of the algebras investigated by J. Cuntz in \cite{Cuntz}, which have had much attention from analysts in the last two decades (see \cite{Raeburn} for an overview of the subject). Leavitt path algebras can also be viewed as a broad generalization of the algebras constructed by W. G. Leavitt in \cite{Leavitt} to produce rings without the Invariant Basis Number property (i.e., whose modules have bases with different cardinality).

The Leavitt path algebra $L_K(E)$ was introduced in 2004 in the papers \cite{AA1} and \cite{AMP}. $L_K(E)$ was first defined for a row-finite graph $E$ (countable graph such that every vertex emits only a finite number of edges) and a field $K$. Despite a relatively recent introduction, they have already generated quite a bit of activity. The main directions of research include: characterization of algebraic properties of a Leavitt path algebra $L_K(E)$ in terms of graph-theoretic properties of $E$; study of the modules over $L_K(E)$; computation of various substructures (such as the Jacobson radical, the center, the socle and the singular ideal); investigation of the relationships with $C^*(E)$ and general $C^*$-algebras; classification programs; study of the $K$-theory; and generalization of the constructions and results first from row-finite to countable graphs and finally, from countable to completely arbitrary graphs.  For examples of each of these directions see for instance \cite{ARV} and the references therein.

In the current paper we focus on the first and last of these lines: concretely, we give a graph-theoretic characterization of the ring-theoretic property of being indecomposable, and we do that in the most general context of arbitrary graphs.

Many times in the literature, the graph-theoretic conditions characterizing some analytic property of the graph $C^*$-algebra (for instance, simplicity in \cite{BPRS} or purely infinite simplicity in \cite{KPR}) turned out to be the same graph-theoretic condition characterizing the corresponding algebraic version of these properties (simplicity in \cite{AA1} and purely infinite simplicity in \cite{AA2}). Even though some efforts are currently being made to obtain some kind of ``Rosetta stone" to transfer information from the analytic to the algebraic world and vice versa, so far this goal has remained elusive. Therefore, the only way to proceed so far is by working out ad hoc methods in either setting (analytic or algebraic) to obtain the desired results. This is the case once more for the property of being indecomposable discussed in the current paper: the analytic result was given in \cite{H} for the $C^*$-algebras $C^*(E)$, and we give here the algebraic analogue for the Leavitt path algebras $L_K(E)$.

The paper is divided as follows. In Section 2 we give the definition of $L_K(E)$, basic properties and main examples. In the next section we prove some preliminary results that will be needed later in the paper; we also recall there the concepts and constructions about arbitrary graphs which will be of use to us throughout the paper. The main result, Theorem \ref{theresult}, is given in Section 4 and it characterizes the indecomposable Leavitt path algebras (i.e., those which cannot be written as a direct sum of nontrivial two-sided ideals). Finally, the last section includes two decomposition results: Corollary \ref{decomposition} and Corollary \ref{twosidednoetherian}, which show that if $L_K(E)$ has certain finiteness conditions (specifically being unital, two-sided noetherian or two-sided artinian), then it can be decomposed uniquely as a direct sum of indecomposable Leavitt path algebras for suitable graphs. The paper finishes off with some examples that illustrate the differences arising when considering arbitrary graphs instead of row-finite ones.


\section{Definition and Examples}

A (directed) {\it graph} $E=(E^{0},E^{1},r,s)$ consists of two sets $E^{0}$ and $
E^{1}$ together with maps $r,s:E^{1}\rightarrow E^{0}$. The elements of $
E^{0}$ are called \textit{vertices} and the elements of $E^{1}$ \textit{edges}. If $s^{-1}(v)$ is a finite set for every $v\in E^{0}$, then the graph is
called \textit{row-finite}.

If a vertex $v$ emits no edges, that is, if $s^{-1}(v)$ is empty, then $v$
is called a\textit{\ sink}. A vertex $v$ is called an \textit{infinite
emitter} if $s^{-1}(v)$ is an infinite set, and $v$ is called a \textit{regular vertex} if $s^{-1}(v)$ is a finite non-empty set. The set of infinite emitters is denoted by $E^0_{\text{inf}}$, and the set of regular vertices is denoted by $E^0_{\text{reg}}$.

All the graphs $E$ that we consider here are arbitrary in the sense that no
restriction is placed either on the number of vertices in $E$ (such as being
a countable graph) or on the number of edges emitted by any vertex (such as
being row-finite).

A path $\mu $ in a
graph $E$ is a finite sequence of edges $\mu =e_{1}\dots e_{n}$ such that $
r(e_{i})=s(e_{i+1})$ for $i=1,\dots ,n-1$. In this case, $n=l(\mu)$ is the length
of $\mu $; we view the elements of $E^{0}$ as paths of length $0$. For any $n\in {\mathbb N}$ the set of paths of length $n$ is denoted by $E^n$. Also, $\text{Path}(E)$ stands for the set of all paths, i.e., $\text{Path}(E)=\bigcup_{n\in {\mathbb N}} E^n$. We denote
by $\mu ^{0}$ the set of the vertices of the path $\mu $, that is, the set $
\{s(e_{1}),r(e_{1}),\dots ,r(e_{n})\}$.

A path $\mu $ $=e_{1}\dots e_{n}$ is \textit{closed} if $r(e_{n})=s(e_{1})$,
in which case $\mu $ is said to be {\it based at the vertex} $s(e_{1})$. The closed path $\mu $ is called a \textit{cycle} if it does not pass through any of its vertices twice, that is, if $s(e_{i})\neq s(e_{j})$ for every $i\neq j$. An \textit{exit }for a path $\mu =e_{1}\dots e_{n}$ is an edge $e$ such that $s(e)=s(e_{i})$ for some $i$ and $e\neq e_{i}$. We say that $E$ satisfies \textit{Condition }(L) if every simple closed path in $E$ has an exit, or, equivalently, every cycle in $E$ has an exit. A graph $E$ is called \emph{acyclic} if it does not have any cycles.

We define a relation $\geq $ on $E^{0}$ by setting $v\geq w$ if there exists
a path $\mu$ in $E$ from $v$ to $w$, that is, $v=s(\mu)$ and $w=r(\mu)$. A subset $H$ of $E^{0}$ is called \textit{hereditary} if $v\geq w$ and $v\in H$ imply $w\in H$. A
set $H \subseteq E^0$ is \textit{saturated} if for any regular vertex $v$, $r(s^{-1}(v))\subseteq H$ implies $v\in H$.

The set $E^{\leq\infty}$ consists of all infinite paths $e_1 e_2 e_3\dots$ together with all paths which end in sinks. A graph is called \emph{cofinal} if for every $p\in E^{\leq\infty}$ and $v\in E^0$, there exists $w\in p^0$ such that $v\geq w$.

For each $e\in E^{1}$, we call $e^{\ast }$ a {\it ghost edge}. We let $r(e^{\ast}) $ denote $s(e)$, and we let $s(e^{\ast })$ denote $r(e)$.

\begin{definition} {\rm Given an arbitrary graph $E$ and a field $K$, the \textit{Leavitt path} $K$\textit{-algebra} $L_{K}(E)$ is defined to be the $K$-algebra generated by a set $\{v:v\in E^{0}\}$ of pairwise orthogonal idempotents together with a set of variables $\{e,e^{\ast }:e\in E^{1}\}$ which satisfy the following
conditions:

(1) $s(e)e=e=er(e)$ for all $e\in E^{1}$.

(2) $r(e)e^{\ast }=e^{\ast }=e^{\ast }s(e)$\ for all $e\in E^{1}$.

(3) (The ``CK-1 relations") For all $e,f\in E^{1}$, $e^{\ast}e=r(e)$ and $
e^{\ast}f=0$ if $e\neq f$.

(4) (The ``CK-2 relations") For every regular vertex $v\in E^{0}$,
\begin{equation*}
v=\sum_{\{e\in E^{1},\ s(e)=v\}}ee^{\ast}.
\end{equation*}}
\end{definition}

An alternative definition for $L_K(E)$ can be given using the extended graph $\widehat{E}$. This graph has the same set of vertices $E^0$ and same set of edges $E^1$ together with the so-called ghost edges $e^*$ for each $e\in E^1$, whose directions are opposite those of the corresponding $e\in E^1$. Thus, $L_K(E)$ can be defined as the usual path algebra $K\widehat{E}$ subject to the Cuntz-Krieger relations (3) and (4) above.

\smallskip

If $\mu = e_1 \dots e_n$ is a path in $E$, we write $\mu^*$ for the element $e_n^* \dots e_1^*$ of $L_{K}(E)$. With this notation it can be shown that the Leavitt path
algebra $L_{K}(E)$ can be viewed as a $K$-vector space span of $\{pq^{\ast } \ \vert \ p,q\,  \hbox{are paths in} \,  E\}$. (The elements of $E^0$ are viewed as paths of length $0$, so that this set includes elements of the form $v$ with $v\in E^0$.)

If $E$ is a finite graph, then $L_{K}(E)$ is unital with $\sum_{v\in E^0} v=1_{L_{K}(E)}$; otherwise, $L_{K}(E)$ is a ring with a set of local units (i.e., a set of elements $X$ such that for every finite collection $a_1,\dots,a_n\in L_K(E)$, there exists $x\in X$ such that $a_ix=a_i=xa_i$) consisting of sums of distinct vertices of the graph.

Many well-known algebras can be realized as the Leavitt path algebra of a graph. The most basic graph configurations are shown below (the isomorphisms for the first three can be found in \cite{AA1}, the fourth in \cite{S}, and the last one in \cite{AA2}).

\begin{examples}\label{examples}{\rm The ring of Laurent polynomials $K[x,x^{-1}]$ is the Leavitt path algebra of the graph given by a single loop graph. Matrix algebras ${\mathbb M}_n(K)$ can be realized by the line graph with $n$ vertices and $n-1$ edges. Classical Leavitt algebras $L_K(1,n)$ for $n\geq 2$ can be obtained by the $n$-rose -- a graph with a single vertex and $n$ loops. Namely, these three graphs are:

$$
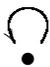
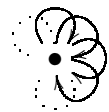
\begin{matrix} \xymatrix{{\bullet} \ar@(ur,ul)} \hskip3cm &
\xymatrix{{\bullet} \ar [r]  & {\bullet} \ar [r]  & {\bullet} \ar@{.}[r] & {\bullet} \ar [r]  & {\bullet} }
\hskip3cm  & \xymatrix{{\bullet} \ar@(ur,dr)  \ar@(u,r)  \ar@(ul,ur)  \ar@{.} @(l,u) \ar@{.} @(dr,dl)
\ar@(r,d) & }
\end{matrix}$$

\medskip

The algebraic counterpart of the Toeplitz algebra $T$ is the Leavitt path algebra of the graph having one loop and one exit: $$
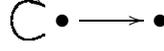
\xymatrix{{\bullet} \ar@(dl,ul) \ar[r] & {\bullet}  }$$

Combinations of the previous examples are possible. For instance, the Leavitt path algebra of the graph

$$
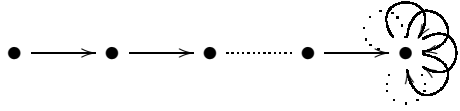
\xymatrix{{\bullet} \ar [r]  & {\bullet} \ar [r]  & {\bullet}
\ar@{.}[r] & {\bullet} \ar [r]  & {\bullet}
 \ar@(ur,dr)  \ar@(u,r)  \ar@(ul,ur)  \ar@{.} @(l,u) \ar@{.} @(dr,dl) \ar@(r,d) & }$$ \smallskip

\noindent is ${\mathbb M}_n(L_K(1,m))$, where $n$ denotes the number of vertices in the graph and $m$ denotes the number of loops.
}
\end{examples}

Another useful property of $L_K(E)$ is that it is a graded algebra, that is, it can be decomposed as a direct sum of homogeneous components $L_K(E)=\bigoplus_{n\in {\mathbb Z}} L_K(E)_n$ satisfying $L_K(E)_nL_K(E)_m\subseteq L_K(E)_{n+m}$. Actually, $$L_K(E)_n=\text{span}_K\{pq^*: p,q\in \text{Path}(E) , l(p)-l(q)=nÊ \}.$$

Every element $x_n\in L_K(E)_n$ is a homogeneous element of degree $n$. An ideal $I$ is graded if it inherits the grading of $L_K(E)$, that is, if $I=\bigoplus_{n\in {\mathbb Z}} (I\cap L_K(E)_n)$.

\section{Preliminary results}


In this paper we will be dealing with arbitrary graphs, and therefore many of the technicalities which arise when considering these graphs must be taken into account. The following concepts and results from \cite{T} will be used in the sequel.

A vertex $w$ is called a \textit{breaking vertex }of a hereditary saturated
subset $H$ if $w\in E^{0}\backslash H$ is an infinite emitter with the property that $0< |s^{-1}(w)\cap r^{-1}(E^{0}\backslash H)|<\infty $. The set of all breaking vertices of $H$ is denoted by $B_{H}$. For any $v\in B_{H}$, $v^{H}$ denotes the element $$v^{H}=v-\sum_{s(e)=v,\ r(e)\notin H}ee^{\ast }.$$ Note that the sum is finite if $v$ is a breaking vertex. Given a hereditary saturated subset $H$ and a subset $S\subseteq B_{H}$, a pair $(H,S)$ is called an \textit{admissible pair }and $I_{(H,S)}$ denotes the ideal generated in $L_K(E)$ by $H\cup \{v^{H}:v\in S\}$. It was shown in \cite{T} that the graded ideals of $L_{K}(E)$ are precisely the ideals of the form $
I_{(H,S)}$ for some admissible pair $(H,S)$. Moreover, it was shown that $
I_{(H,S)}\cap E^{0}=H$ and $\{v\in B_{H}:v^{H}\in I_{(H,S)}\}=S$.

Given an admissible pair $(H,S)$, the corresponding \textit{quotient graph} $
E\backslash (H,S)$ is defined as follows:
\begin{align*}
(E\backslash (H,S))^{0}& =(E^{0}\backslash H)\cup \{v^{\prime }:v\in
B_{H}\backslash S\}; \\
(E\backslash (H,S))^{1}& =\{e\in E^{1}:r(e)\notin H\}\cup \{e^{\prime }:e\in
E^{1},r(e)\in B_{H}\backslash S\}.
\end{align*}%
Further, $r$ and $s$ are extended to $(E\backslash (H,S))^{0}$ by setting $
s(e^{\prime })=s(e)$ and $r(e^{\prime })=r(e)^{\prime }$. Note that, in the
graph $E\backslash (H,S)$, the vertices $v^{\prime }$ are all sinks.

The result \cite[Theorem 5.7]{T} states that there is an epimorphism $\phi
:L_{K}(E)\rightarrow L_{K}(E\backslash (H,S))$ with $\ker \phi =$ $I_{(H,S)}$
and that $\phi (v^{H})=v^{\prime }$ for $v\in B_{H}\backslash S$. Thus $L_{K}(E)/I_{(H,S)}\cong L_{K}(E\backslash (H,S))$. This theorem has been established in \cite{T} under the hypothesis that $E$ is a graph with at most countably many vertices and edges; however, an examination of the proof reveals that the countability condition on $E$ is not utilized. So \cite[Theorem 5.7]{T} holds for arbitrary graphs $E$.

Recall that an $K$-algebra $R$ is said to be \emph{indecomposable} if it does not have any decomposition into a direct sum of two nonzero ideals $I$ and $J$ as $R=I\oplus J$. Of course $R$ is called \emph{decomposable} if it is not indecomposable, that is, if there exists at least one such decomposition.

\begin{remark} \label{conectedness} {\rm  In this paper we will be concerned with indecomposable Leavitt path algebras, and we will characterize them in terms of properties of the graph. We note that sometimes in the literature the term ``indecomposable algebra" is equivalent to ``connected algebra" (see for instance \cite[Lemma 1.7]{ASS} where it is stated that for a finite quiver $Q$, the path algebra $KQ$ is connected if and only if the quiver $Q$ is connected).

However we will stick to the first terminology since the second one could be misleading in the sense that the connected Leavitt path algebras are not necessarily the ones whose underlying graph is connected. Actually we can prove the following: $$ E \text{ directly connected } \   \begin{smallmatrix} \Longrightarrow \\ \not\Longleftarrow \end{smallmatrix} \ L_K(E) \text{ connected } \ \begin{smallmatrix} \Longrightarrow \\ \not\Longleftarrow \end{smallmatrix} \  E \text{ connected,}$$ that is, the implications cannot be reversed.

Recall that \emph{directly connectedness} takes into account the direction of the edges (that is: for every $v,w\in E^0$ there exists a path $\mu$ with $s(\mu)=v$ and $r(\mu)=w$), while \emph{connectedness} does not (i.e., for every $v,w\in E^0$ there exists a \emph{walk} from $v$ to $w$, that is, a path in the extended graph $\widehat{E}$ which includes the ghost edges $e^*\in (E^1)^*$ in the reverse direction to any real edge $e\in E^1$).

If $E$ is not connected, then we have two disjoint subgraphs $E_1$ and $E_2$ such that $E$ is the disjoint union $E=E_1\coprod E_2$. An application of \cite[Proposition 2.4]{AC} gives that $L_K(E)=L_K(E_1)\oplus L_K(E_2)$ so that $L_K(E)$ is not connected (i.e., it is decomposable).

To see that the reverse implication is false it is enough to consider the graph $$\xymatrix{{\bullet} & {\bullet} \ar[l]Ê\ar[r] & {\bullet}}$$ which is connected but its Leavitt path algebra is, by \cite[Proposition 3.5]{AAS}, isomorphic to $L_K(E)\cong {\mathbb M}_2(K)\oplus {\mathbb M}_2(K)$, which is not connected.

Now assume that $E$ is directly connected. In this case the graph must be cofinal (clearly, for every path $p\in E^{\leq\infty}$ and every $v\in E^0$, there exists $\mu\in \text{Path}(E)$ and $w\in p^0$ such that $v=s(\mu)$ and $w=r(\mu)$). So we distinguish two cases. First, if there exists a cycle $c$ which does not have an exit then we get: there do not exist vertices outside $c$ (by the directly connectedness); and also there are no further edges beside those of $c$ (otherwise we would have an exit). In this case \cite[Theorem 3.3]{AAS2} gives that $L_K(E)\cong {\mathbb M}_n(K[x,x^{-1}])$ where $n$ is the length of the cycle $c$. This algebra is unital and does not have central idempotents besides $0$ and $1$, hence it is connected. In the second case, that is, if every cycle has an exit, then \cite[Theorem 3.11]{G2} and \cite[Lemma 2.8]{APS} yield that $L_K(E)$ is simple, and therefore also connected.

The other implication does not hold as can be seen for instance with the following graph $$\xymatrix{{\bullet}  \ar[r]Ê&{\bullet} \ar[r] & {\bullet}}$$ whose Leavitt path algebra is $L_K(E)\cong {\mathbb M}_3(K)$, which is connected even though the graph  is not directly connected.
}
\end{remark}

\begin{remark}\label{centraidempotents}{\rm In the literature the term ``indecomposable" also appears as an equivalent statement to ``there are no nontrivial central idempotents". However once more we will only use the first terminology since in our setting these two statements are not equivalent, due to the lack of an identity. An easy example is given by the graph: $$
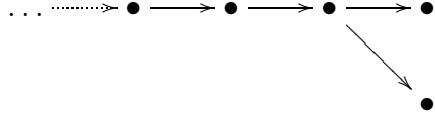
\xymatrix{\dots\ar@{.>}[r] & \bullet \ar[r] &  \bullet \ar[r] & \bullet \ar[r] \ar[dr] & \bullet \\  & & & & \bullet}$$ whose Leavitt path algebra is $L_K(E)\cong {\mathbb M}_\infty (K)Ê\oplus {\mathbb M}_\infty (K)$ by \cite[Proof of Proposition 2.4]{AAPS}. Then $L_K(E)$ is not indecomposable even though it does not have nonzero central idempotents (actually the center of $L_K(E)$ is zero).
}
\end{remark}

It is well-known that the following result is true for the row-finite case of $E$ (see \cite[Proof of Proposition 5.2 and Theorem 5.3]{AMP}). We thank P. Ara and E. Pardo for pointing out to us how to prove that this result remains valid for arbitrary graphs.

\begin{proposition} [\cite{Aprivate,Pprivate}]\label{gradedidempotents} Let $K$ be a field and $E$ an arbitrary graph. Then the following statements are equivalent for a two-sided ideal $I$ of $L_K(E)$:
\begin{enumerate}
\item $I$ is graded.
\item $I$ is generated by idempotents.
\end{enumerate}
\end{proposition}
\begin{proof} The transition to arbitrary graphs can be done in several ways. The first one is via the monoid $V(L_K(E))$ of the finitely generated projective modules of $L_K(E)$. The proofs of \cite[Proposition 5.2 and Theorem 5.3]{AMP} make use of the fact that $V(L_K(E))$ is an unperforated, separative and refinement monoid when $E$ is row-finite. However, Ken Goodearl extended those monoid properties to completely arbitrary graphs in \cite[Theorem 5.8]{G2}, which in turn paves the way for generalizing the aforementioned results for arbitrary graphs.

A second approach is via the Leavitt path algebras of separated graphs (these graphs are generalizations of the usual directed graphs which, considering the trivial partitions, allow us to recover the classical construction of $L_K(E)$). Now, in \cite[Section 6]{AG} the ideals generated by idempotents are described, and from these results one obtains that the ideals generated by idempotents are actually generated by idempotents which are homogeneous of zero degree, and therefore graded. Conversely, even in the arbitrary graph case, every graded ideal is generated by the $0$-component turns out to be locally matricial and therefore von Neumann regular. Thus the $0$-component is generated by idempotents.
\end{proof}

\begin{remark} \label{lattice} {\rm In the following results, we need to use the structure of, and operations on,
the lattice ${\mathcal L}_E$ of the admissible pairs of $E$, originally described in
\cite[Definition 5.4]{T}. However, Mark Tomforde informed us that corrections needed to be made to those operations as described, and
provided us with the corrections. These have not yet appeared in
the literature, so we include them here, and thank Mark Tomforde
for providing them.
 $$(X,S_Y)\wedge (Y,S_Y):=((X\cap Y),((S_X\cup X)\cup (S_Y\cup Y))\cap B_{X\cap Y})$$ $$(X,S_X)\vee (Y,S_Y):=\left(\bigcup_{n=0}^\infty X_n, (S_X\cup S_Y)\cap B_{\cup_{n=0}^\infty X_n}\right)$$ where $X_0:=X\cup Y$ and $$X_{n+1}:=X_n\cup \{v\in E^0_{\text{reg}}: r(s^{-1}(v))\subseteq X_n\} \cup \{v\in S_X\cup S_Y: r(s^{-1}(v))\subseteq X_n\}.$$}
\end{remark}

The following is an easy lemma that states some useful properties about breaking vertices which we will be using throughout the paper.

\begin{lemma}\label{intersectionB} Let $E$ be a graph and $X,Y\in {\mathcal H}_E$ two hereditary and saturated sets of vertices.
\begin{enumerate}
\item If $X=E^0$ or $X=\emptyset$, then $B_X=\emptyset$.
\item If $X\cap Y=\emptyset$, then $B_X\cap B_Y=\emptyset$.
\end{enumerate}
\end{lemma}
\begin{proof}
(i) follows trivially from the definition of a set of breaking vertices: $$B_X=\{w\in E^0_{\text{inf}}\setminus XÊ\ :\ 0<|s^{-1}(w)\cap r^{-1}(E^0\setminus X)|<\infty \}.$$

(ii). Let $v\in B_X\cap B_Y$. Since $v\in B_X$, then $0<|s^{-1}(v)\cap r^{-1}(E^0\setminus X)|<\infty$, which in turn implies, since $v$ is an infinite emitter, that $|s^{-1}(v)\cap r^{-1}(X)|=\infty$. On the other hand, $v\in B_Y$ gives $0<|s^{-1}(v)\cap r^{-1}(E^0\setminus Y)|<\infty$. But now $X\cap Y=\emptyset$ implies that $s^{-1}(v)\cap r^{-1}(X)\subseteq s^{-1}(v)\cap r^{-1}(E^0\setminus Y)$, a contradiction.
\end{proof}

We now prove a proposition which is the first step towards our characterization of the indecomposable Leavitt path algebras.

\begin{proposition}\label{S_1=B_X} Let $E$ be a graph and let $(X,S_1)$ and $(Y,S_2)$ two admissible pairs in $E$ such that $L_K(E)=I_{(X,S_1)}\oplus I_{(Y,S_2)}$. Then
$S_1=B_X$ and $S_2=B_Y$.
\end{proposition}
\begin{proof}
We will show that $S_1=B_X$ and then $S_2=B_Y$ will be analogous. With the lattice operation detailed in Remark \ref{lattice} above, we have a lattice isomorphism between the set of graded ideals ${\mathcal L}_{gr}(L_K(E))$ and the lattice of admissible pairs ${\mathcal H}_E$, see \cite[Theorem 5.7 (1)]{T}. Therefore we have $$L_K(E)=I_{(X,S_1)}\oplus I_{(Y,S_2)}=I_{(X,S_1)\vee (Y,S_2)}=I_{(\bigcup_{n=0}^\infty X_n, (S_1\cup S_2)\cap B_{\cup_{n=0}^\infty X_n})},$$ where the $X_n$'s are those of Remark \ref{lattice}.

Now suppose that there is a vertex $v\in E^0$ such that $v\in B_X\setminus S_1$. Recall that we always have that $X=\{w\in E^0\ |\ w\in I_{(X,S_1)}\}$ (see for instance the proof of \cite[Theorem 5.7 (1)]{T}), and therefore $v\in \bigcup_{n=0}^\infty X_n$. We will show, step by step, that this situation is not possible.

\underline{Base Case}: $v\not\in X_0$. Suppose on the contrary that $v\in X_0=X\cup Y$. Recall that the set of breaking vertices of $X$ is given by $$B_X=\{w\in E^0_{\text{inf}}\setminus XÊ\ :\ 0<|s^{-1}(w)\cap r^{-1}(E^0\setminus X)|<\infty \},$$ and so since by hypothesis $v\in B_X$, we must have that $v\not\in X$ by the first condition defining the set $B_X$, which in turn implies that $v\in Y$.

On the other hand, the second condition, together with the fact that $v$ must be an infinite emitter, show that there exists an edge $e\in E^1$ such that $v=s(e)$ and $r(e)\in X$. Now apply that  $Y$ is a hereditary subset to obtain that $r(e)\in Y$, a contradiction with the fact that $X$ and $Y$ are disjoint.

\underline{Inductive Step}: $v\not\in X_n$ implies $v\not\in X_{n+1}$. We proceed again by contradiction and suppose that $v\in X_{n+1}$, which we recall is the union of three different sets of vertices $$X_{n+1}:=X_n\cup \{v\in E^0_{\text{reg}}: r(s^{-1}(v))\subseteq X_n\} \cup \{v\in S_1\cup S_2: r(s^{-1}(v))\subseteq X_n\}.$$ By the induction hypothesis $v$ does not belong to the first set, and it also does not belong to the second set as $v$ is an infinite emitter. Thus, $v$ must belong to the third set.

Now use the hypothesis that $v\not\in S_1$ to get that $v\in S_2\subseteq B_Y$. And hence $v\in B_X\cap B_Y=\emptyset$, by Lemma \ref{intersectionB}, a contradiction.
\end{proof}


\section{Decomposable Leavitt path algebras}

We now give a characterization of decomposable Leavitt path algebras in terms
only of the underlying graph. In this section, we state our main
theorem, then prove it in two parts. The theorem is the Leavitt
path algebra sibling of the corresponding $C^*$-algebra result given by
Hong in \cite[Theorem 4.1]{H}.

\begin{definition}\label{def} {\rm Let $E$ be an arbitrary graph and $X\in {\mathcal H}_E$ be a saturated hereditary subset. An $X$-compatible path is a directed path of positive length $\alpha=\alpha_1\dots\alpha_n$ satisfying the following two conditions:

\begin{enumerate}
\item[{\rm (i)}] $r(\alpha_n)\in X$.

\item[{\rm (ii)}] $s(\alpha_n)\not\in X\cup B_X$.
\end{enumerate}
}
\end{definition}

Note that the notion of $X$-compatible path, as defined above, arises in other parts of the graph algebra
literature: If $X$ is a saturated hereditary subset, the ideal $I_X$ is isomorphic
to the graph algebra of a graph $E$, and the construction of $E$ uses the $X$-compatible paths defined above (see the set $F_1(X,\emptyset)$ in \cite[Definition 4.1]{Ruiz}).\\

\begin{theorem}\label{theresult} Let $E$ be an arbitrary graph. The Leavitt path algebra $L_K(E)$ is decomposable if and only if there exist nontrivial hereditary and saturated subsets $X,Y\in {\mathcal H}_E$ such that $X\cap Y=\emptyset$ and, for every  $v\in E^0\setminus (X\cup Y)$, there exists at least one but finitely many paths starting at $v$ that are either $X$-compatible paths or $Y$-compatible paths.
\end{theorem}

We split the proof of the theorem into several pieces that we will later put together.

\begin{lemma}\label{intersectionB} Let $E$ be a graph, $X,Y\in {\mathcal H}_E$ two disjoint hereditary and saturated sets of vertices and $\mu=\alpha_1\dots\alpha_n\in \text{Path}(E)$. Then $\mu$ is either $X$-compatible or $Y$-compatible if and only if $\mu$ satisfying the properties:
\begin{enumerate}
\item[{\rm (i)}] $\alpha_i\not\in \{\alpha\in E^1 \ |\ s(\alpha)\in B_X, r(\alpha)\in X \}\cup \{\alpha\in E^1 \ |\ s(\alpha)\in B_Y, r(\alpha)\in Y \} $ and $\alpha_i\in E^1$ for all $i=1,\cdots,n$.

\item[{\rm (ii)}] $r(\alpha_n)\in X\cup Y$, and $s(\alpha_i)\not\in X\cup Y$ for all $i=1,\dots,n$.
\end{enumerate}
\end{lemma}

\begin{proof}
Assume that $\mu$ satisfies the properties (i) and (ii), then Condition (ii)
implies either $r(\alpha_n)\in X$ or $r(\alpha_n)\in Y$. If $r(\alpha_n)\in X$, then
Condition (i) implies $\mu$ is $X$-compatible. If
$r(\alpha_n)\in Y$, then Condition (i) implies $\mu$ is $Y$-compatible.
Conversely, suppose $\mu$ is a path that is either $X$-compatible or $Y$-compatible. If $\mu$ is $X$-compatible, then $r(\alpha_n)\in X$ by condition (i) of \ref{def} and
$s(\alpha_n)\not\in X$ by condition (ii) of \ref{def}, and the fact $X$ is hereditary implies $s(\alpha_i)\not\in X$
for all $i$. Likewise, by the fact $Y$ is hereditary and $X$ is disjoint from $Y$,
$r(\alpha_n)\in X$ implies that $s(\alpha_i)\not\in Y$ for all $i$. Thus property (ii) holds. In addition, if $\mu$ is $X$-compatible, then condition (ii) of \ref{def} implies $s(\alpha_n)\not\in B_X$. Since no edge other than $\alpha_n$ has range in $X\cup Y$, it follows that the property
(i) holds. Likewise, a $Y$-compatible
path satisfies the properties (i) and (ii).
\end{proof}

\begin{definition} {\rm Let $E$ be an arbitrary graph and let $X,Y\in {\mathcal H}_E$ be hereditary and saturated subsets such that $X\cap Y=\emptyset$. A path $\mu=\alpha_1\dots\alpha_k\in \text{Path}(E)$ that is either $X$-compatible or $Y$-compatible is call an \emph{$XY$-compatible path}. The edges $\alpha_i\in E^1$ constituting the path $\alpha$ are called \emph{$XY$-compatible edges}.}
\end{definition}

\begin{proposition}\label{firstimplication} Let $E$ be an arbitrary graph containing two nontrivial hereditary and saturated subsets $X,Y\in {\mathcal H}_E$ such that $X\cap Y=\emptyset$ and, for every $v\in E^0\setminus (X\cup Y)$, there exists at least one but finitely many $XY$-compatible paths starting at $v$. Then $L_K(E)=I_{(X,B_X)}\oplus I_{(Y,B_Y)}$.
\end{proposition}

\begin{proof}
The sum is direct because of the lattice isomorphism between admissible pairs and graded ideals. We have: $I_{(X,B_X)}\cap I_{(Y,B_Y)}=I_{(X,B_X)\wedge (Y,B_Y)}$; and we use Remark \ref{lattice} to get that $(X,B_X)\wedge (Y,B_Y)=(X\cap Y,Z)=(\emptyset,Z)$ for some subset $Z$ that satisfies $Z\subseteq B_{X\cap Y}=B_\emptyset=\emptyset$ (by Lemma \ref{intersectionB} (1)). Thus, taking into account the way the ideal $I_{(H,S)}$ is defined, we obtain $I_{(X,B_X)}\cap I_{(Y,B_Y)}=I_{(\emptyset,\emptyset)}=0$.

Define now ${\mathcal K}:=I_{(X,B_X)}\oplus I_{(Y,B_Y)}$ and we will show that ${\mathcal K}=L_K(E)$. Recall that $I_{(X,B_X)}\cap E^0=X$ and thus $X\cup Y\subseteq {\mathcal K}$. On the other hand, since ${\mathcal K}$ is a two-sided ideal of $L_K(E)$, it is enough to show that ${\mathcal K}$ contains all the vertices in $E^0$ (since the set of sums of distinct vertices of the graph always constitutes a set of local units for $L_K(E)$). Therefore, our task will be to show that $v\in {\mathcal K}$ for every $v\in E^0\setminus (X\cup Y)$. Assume we have picked such $v$ and proceed by induction on the number $n$ of different $XY$-compatible paths which start at that $v$ (our hypothesis says that such $n$ exists and is nonzero).

\underline{Base Case}: $n=1$. Let $\mu=\alpha_1\dots \alpha_k$ be the only $XY$-compatible path starting at $v$. We have $r(\alpha_k)\in X\cup Y\subseteq {\mathcal K}$, and we will be proceeding backwards along the path $\mu$ proving inductively that $s(\alpha_k)\in {\mathcal K}$ first, then $s(\alpha_{k-1})\in {\mathcal K}$, and so on until getting $s(\alpha_1)=v\in {\mathcal K}$. We distinguish cases:

$\circ$ Case 1: $s(\alpha_k)\not\in B_X\cup B_Y$. In this situation we have that $\alpha_k$ is the only edge emitted by $s(\alpha_k)$. Otherwise, suppose that there exists $e\in E^1\setminus\{\alpha_k\}$. Two things can happen: if $r(e)\in X\cup Y$ then $\nu:=\alpha_1\dots\alpha_{k-1}e$ is a $XY$-compatible path starting at $v$ (because $s(e)\not\in B_X\cup B_Y$). This new path is different from $\mu$, a contradiction with our hypothesis.

On the other hand, if $r(e)\not\in X\cup Y$, then by hypothesis there exists a $XY$-compatible path $\sigma\in \text{Path}(E)$ starting at $r(e)$, and again since $s(e)\not\in B_X\cup B_Y$ then $\alpha_1\dots\alpha_{k-1}e\sigma$ is a $XY$-compatible path starting at $v$, which is different from $\mu$.

Therefore we have $s^{-1}(s(\alpha_k))=\{\alpha_k\}$, and an application of the (CK2) relation gives that $$s(\alpha_k)=\alpha_k\alpha_k^*=\alpha_k\ r(\alpha_k)\ \alpha_k^*\in {\mathcal K} ,$$ as needed.

$\circ$ Case 2: $s(\alpha_k)\in B_X$. We will prove that in this case $\alpha_k$ is the only edge starting at $s(\alpha_k)$ whose range lies outside $X$. Suppose on the contrary that there exists $f\in E^1$ such that $f\neq \alpha_k$, $s(f)=s(\alpha_k)$ and $r(f)\not\in X$.

Two things can happen: if $r(f)\in Y$ then $\alpha_1\dots \alpha_{k-1}f$ is a $XY$-compatible path starting at $v$, and this cannot happen as $\mu$ was the only such path. In case $r(f)\not\in Y$, then there exists a $XY$-compatible path $\sigma\in \text{Path}(E)$ starting at $r(f)$, and because $f$ is a $XY$-compatible edge, we get $\alpha_1\dots\alpha_{k-1}f\sigma$ is a $XY$-compatible path starting at $v$, which is different from $\mu$.

Thus, $s^{-1}(s(\alpha_k))\cap r^{-1}(E^0\setminus X)=\{\alpha_k\}$ and this yields $s(\alpha_k)^{X}=s(\alpha_k)-\alpha_k\alpha_k^*$. Now, taking into account that $I_{(X,B_X)}$ is generated as a two-sided ideal by $X\cup \{v^X : v\in B_X \}$ we get $$s(\alpha_k)=s(\alpha_k)^{X}+\alpha_k\ r(\alpha_k)\ \alpha_k^*\in I_{(X,B_X)} + L_K(E)\ {\mathcal K} \ L_K(E)\subseteq {\mathcal K},$$ as we wanted to prove.

$\circ$ Case 3: $s(\alpha_k)\in B_Y$. Analogous to Case 2.

Thus, in any case we obtain $s(\alpha_k)\in {\mathcal K}$. Now proceeding in the same fashion with $\alpha_{k-1}$ by distinguishing the three cases above with the vertex $s(\alpha_{k-1})$, we would conclude that $s(\alpha_{k-1})\in {\mathcal K}$, and repeating this process $k$ times we eventually reach $v$, hence proving that $v\in {\mathcal K}$.

\underline{Induction Step}: We now assume that whenever $w\in E^0\setminus(X\cup Y)$ is such that there exists at most $n$ different $XY$-compatible paths starting at $w$, then necessarily $w\in {\mathcal K}$.

Suppose then that the vertex $v\in E^0\setminus (X\cup Y)$ is such that there exists exactly $n+1$ different $XY$-compatible paths starting at $v$. We need to show that in this situation $v\in {\mathcal K}$.

Let $\mu=\alpha_1\dots \alpha_k$ be an $XY$-compatible path starting at $v$, and choose $m$ to be the minimal index $i$ such that the vertex $s(\alpha_i)$ is the base of at least two different $XY$-compatible edges (such vertex exists because otherwise $\mu$ would be the only $XY$-compatible path starting at $v$). We note that the following property holds: $$ \text{If $g$ is a $XY$-compatible edge such that $s(g)=s(\alpha_m)$, then necessarily $r(g)\in {\mathcal K}$}.\ \ (\dag)$$ Indeed, if $r(g)\in X\cup Y$, then immediately $r(g) \in {\mathcal K}$; and in case $r(g)\not\in X\cup Y$, then there exists at most $n$ different $XY$-compatible paths starting at $r(g)$ (each of them would provide an $XY$-compatible path starting at $v$ just by appending $\alpha_1\dots \alpha_{m-1}g$ on the left and, together with $\mu$, we cannot have more than $n+1$ such paths). Thus, the induction hypothesis applies at $r(g)$ and we obtain $r(g)\in {\mathcal K}$.

We now distinguish two cases:

$\circ$ Case 1: $s(\alpha_m)\not\in B_X\cup B_Y$. In this case we have that every $\alpha\in E^1$ with $s(\alpha)=s(\alpha_m)$ is a $XY$-compatible edge, and therefore we must have finitely many such edges (otherwise we would have infinitely many $XY$-compatible paths starting at $v$). Therefore $$s(\alpha_m)=\sum_{\{f\in E^1\ :\ s(f)=s(\alpha_m)\}}^{\text{finite}} ff^*=\sum_{s(f)=s(\alpha_m)} f\ r(f)\ f^*\in {\mathcal K},$$ where we have used property $(\dag)$ above in the last step.

$\circ$ Case 2: $s(\alpha_m)\in B_X$ (the case $s(\alpha_m)\in B_Y$ is similar). Here we have that the set $\{f\in E^1\ :\ s(f)=s(\alpha_m) \text{ and } r(f)\not\in X\}$ must be finite (note that they constitute $XY$-compatible edges and argue as in Case 1). On the other hand, since $s(\alpha_m)\in B_X$, we know that $s(\alpha_m)^X\in I_{(X,B_X)}$ and all this gives: \begin{align*} s(\alpha_m) =& s(\alpha_m)^X+\sum_{\{f\in E^1\ :\ s(f)=s(\alpha_m) \text{ and } r(f)\not\in X\}}^{\text{finite}} ff^*= \\ = & s(\alpha_m)^X + \sum_{s(f)=s(\alpha_m),\ r(f)\not\in X} f\ r(f)\ f^*\in I_{(X,B_X)} + L_K(E)\ {\mathcal K}L_K(E) \subseteq {\mathcal K}, \end{align*} where we have made use of property $(\dag)$ in the last line.

Now in either case we get that $s(\alpha_m)\in {\mathcal K}$ and the minimality of $m$ implies that there exists only one path $\sigma$ from $v$ to $s(\alpha_m)$ entirely made of $XY$-compatible edges. In this situation we walk backwards along $\sigma$ as in the Base Case above proving that $s(\alpha_{m-1})\in {\mathcal K}$, and after $m$ steps arriving at $s(\alpha_1)=v\in  {\mathcal K}$. This finishes the proof.
\end{proof}

\begin{proposition} \label{secondimplication} Let $E$ be an arbitrary graph such that the Leavitt path algebra $L_K(E)$ is decomposable. Then there exist two nontrivial hereditary and saturated subsets $X,Y\in {\mathcal H}_E$ such that $X\cap Y=\emptyset$ and, for every $v\in E^0\setminus (X\cup Y)$, there exists at least one but finitely many $XY$-compatible paths starting at $v$.
\end{proposition}

\begin{proof}
If $L_K(E)$ is decomposable, then there exist two-sided ideals $I$ and $J$ such that $L_K(E)=I\oplus J$. We first find a set of local units in $I$. Recall that $L_K(E)$ has such a set of local units: the set consisting of finite sums of vertices of $E$. Thus each local unit $e\in L_K(E)$ can be decomposed uniquely as $e=e_I+e_J$ where $e_I\in I$ and $e_J\in J$. Also, since $e=e^2$ we get $e=e_I+e_J=e_I^2+e_J^2+e_Ie_J+e_Je_I=e_I^2+e_J^2$, and this implies $e_I^2-e_I=e_J-e_J^2\in I\cap J=0$, so $e_I$ and $e_J$ are idempotents in $I$ and $J$ respectively.

Now consider ${\mathcal X}=\{e_I\ : \ e\in L_K(E) \text{ is a local unit in }L_K(E)  \}$. This is a set of local units in $I$. To see this, take $x_1,\dots,x_n\in I$ and let $e\in L_K(E)$ be a local unit for the $x_k$ in $L_K(E)$ so that $ex_k=x_k=x_ke$ for all $k=1,\dots, n$. Then since $I$ and $J$ are orthogonal, we get $e_Ix_k=x_k=x_ke_I$ also. This shows that $I$ is generated by idempotents (i.e., the set of local units), so that an application of Proposition \ref{gradedidempotents} yields that $I$ is a graded ideal and similarly so is $J$.

We then know that there exist admissible pairs $(X,S_1)$ and $(Y,S_2)$ such that $I=I_{(X,S_1)}$ and $J=I_{(Y,S_2)}$ and thus Proposition \ref{S_1=B_X} applies to give $B_X=S_1$ and $B_Y=S_2$. Therefore, the lattice given in Remark \ref{lattice} gives that in this situation $E^0=\bigcup_{n=0}^\infty X_n$ where $X_0:=X\cup Y$ and $$X_{n+1}:=X_n\cup \{v\in E^0_{\text{reg}}: r(s^{-1}(v))\subseteq X_n\} \cup \{v\in B_X\cup B_Y: r(s^{-1}(v))\subseteq X_n\}.$$ Let $v\in E^0\setminus (X\cup Y)$. First we will show that there exists at least one $XY$-compatible path starting at $v$.

\smallskip

\underline{Step 1}: $v\in X_1\setminus X_0$. By the construction of $X_1$ we can have two possibilities. First, suppose $v\in E^0_{\text{reg}}$ and $r(s^{-1}(v))\subseteq X_0$. In this case there exists $e\in E^1$ such that $v=s(e)$ and $r(e)\in X_0=X\cup Y$ so the edge $e$ itself constitutes an $XY$-compatible path starting at $v$.

Second, suppose that $v\in B_X$ is such that $r(s^{-1}(v))\subseteq X_0$ (the case that $v\in B_Y$ is analogous). Since $v\in B_X$ then there exists $f\in E^1$ such that $r(f)\not\in X$. So on the one hand $f\not\in\{\alpha\in E^1 : s(\alpha)\in B_X, r(\alpha)\in X\}$, but we also have that $v=s(f)\not\in B_Y$ (since $B_X\cap B_Y=\emptyset$ by Lemma \ref{intersectionB}), and therefore $f\not\in\{\alpha\in E^1 : s(\alpha)\in B_Y, r(\alpha)\in Y\}$. But by hypothesis we also have that $r(f)\in X_0=X\cup Y$. All this then proves that $f$ is an $XY$-compatible path starting at $v$, as needed.

\smallskip

\underline{Step 2}: $v\in X_2\setminus X_1$. Again we have two cases for the vertex $v$. In the first case $v\in E^0_{\text{reg}}$ and $r(s^{-1}(v))\subseteq X_1$. Choose then $e\in E^1$ such that $v=s(e)$ and $r(e)\in X_1$. We can assume that $r(e)\not\in X_0$: otherwise we would have that $r(s^{-1}(v))\subseteq X_0$, and this would imply that $v\in X_1$. Apply now Step 1 to the vertex $w:=r(f)\in X_1\setminus X_0$ to find a $XY$-compatible path $\mu$ starting at $w$. Then $\nu:=e\mu$ is the desired $XY$-compatible path starting at $v$.

In the second case we have that $v\in B_X$ is such that $r(s^{-1}(v))\subseteq X_1$ (for $v\in B_Y$ we would reason in a similar way). Now $v\in B_X$ implies that there exists $f\in E^1$ such that $w:=r(f)\not\in X$. Actually we can assume that $w\not\in X_0$ because otherwise $r(s^{-1}(v))\subseteq X_0$, which would again lead us to the contradiction that $v\in X_1$. Thus, we again apply Step 1 to the vertex $w\in X_1\setminus X_0$ to find a $XY$-compatible path $\mu$ starting at $w$. We now claim that $\nu:=f\mu$ is a $XY$-compatible path starting at $v$. Indeed, as $w\not\in X$ we get that $f\not\in\{\alpha\in E^1 : s(\alpha)\in B_X, r(\alpha)\in X\}$. Furthermore, $v=s(f)\not\in B_Y$ (since $B_X\cap B_Y=\emptyset$ by Lemma \ref{intersectionB}), and therefore $f\not\in\{\alpha\in E^1 : s(\alpha)\in B_Y, r(\alpha)\in Y\}$, as we wanted.

This process can go on by induction for any $n\in {\mathbb N}$ and $v\in X_n\setminus X_{n-1}$, so that the existence of the $XY$-compatible path is proved. We now focus on the finiteness of the set of such compatible paths.

\smallskip

\underline{Step 1}: $v\in X_1\setminus X_0$. We have two options for $v$. Let us suppose first that: $v\in E^0_{\text{reg}}$ and $r(s^{-1}(v))\subseteq X_0$. Then, for every $e\in E^1$ such that $v=s(e)$ we have that its range arrives at $X_0=X\cup Y$, and by the definition of $XY$-compatible path, we see that no edge can follow $e$ to form such a path. Therefore only the edges of $s^{-1}(v)$ can constitute a $XY$-compatible path starting at $v$, but we have finitely many of them.

In the second case we assume that $v\in B_X$ is such that $r(s^{-1}(v))\subseteq X_0$ (if $v\in B_Y$ we would proceed analogously). Since $v$ is a breaking vertex of $X$ this means that $s^{-1}(v)$ is the disjoint union of the infinite set $I$ of edges whose range lie in $X$ and the finite set $F$ of edges whose range do not lie in $X$. That is, $I\subseteq \{\alpha\in E^1 : s(\alpha)\in B_X, r(\alpha)\in X\}$, and therefore the edges of $I$ cannot belong to any $XY$-compatible path: only the edges of $F$ might. Actually they do because $r(s^{-1}(v))\subseteq X_0=X\cup Y$ by assumption, and $v=s(f)\not\in B_Y$ for any $f\in F$, as we have seen in the previous paragraphs.

\smallskip

\underline{step 2}: $v\in X_2\setminus X_1$. Suppose first that $v\in E^0_{\text{reg}}$ and $r(s^{-1}(v))\subseteq X_1$. Split the set as $$s^{-1}(v)=\{g\in E^1: v=s(g), r(g)\in X_1\setminus X_0\}\cup \{g\in E^1: v=s(g), r(g)\in X_0=X\cup Y\},$$ then the elements of the second set are $XY$-compatible edges starting at $v$, while the elements of the first set are such that their range vertices are in the conditions of Step 1 above, so that for each vertex $r(g)$ where $g$ belongs to the first set, there exist at most finitely many $XY$-compatible paths starting at $s(g)$. Considering all possible combinations we obtain only finitely many possible $XY$-compatible paths starting at $v$.

In the second situation we are given $v\in B_X$ with $r(s^{-1}(v))\subseteq X_1$ ($v\in B_Y$ is similar). Here, the edges $\alpha\in E^1$ such that $v=s(\alpha)$ and  $r(\alpha)\in X$ cannot be the first edge of any $XY$-compatible path starting at $v$ so we must focus on the edges $\beta\in E^1$ such that $v=s(\beta)$ and  $r(\beta)\not\in X$, of which only finitely many exist.

Pick one such edge $g$ and we have two cases for $w=r(g)$. If $w\in X_0$ then the edge $g$ is a $XY$-compatible edge and path starting at $v$ which cannot be enlarged since $w\in X\cup Y$. On the other hand, if $w\in X_1\setminus X_0$, then we apply the previous Step 1 to get that there exist finitely many $XY$-compatible paths starting at that particular $w$. Thus, adding up and multiplying all possible combinations, in the end only finitely many $XY$-compatible paths starting at $v$ can be found, as we wanted to prove.

This procedure continues by induction for any $n\in {\mathbb N}$ and $v\in X_n\setminus X_{n-1}$. Thus the finiteness of the set of the $XY$-compatible path is hence shown.
\end{proof}

We now have both pieces to prove our main result.

\begin{proof}[Proof of Theorem \ref{theresult}]  The theorem is now a consequence of Proposition \ref{firstimplication} and Proposition \ref{secondimplication}.
\end{proof}


\section{Corollaries and Examples}

When the graphs have a finite number of vertices (but they can have infinitely many edges or even vertices which are infinite emitters), then a complete decomposition of $L_K(E)$ into unique indecomposable Leavitt path algebras can be achieved. This is done in the following result, which parallels \cite[Corollary 4.2]{H}.

\begin{corollary}\label{decomposition} Let $E$ be an arbitrary graph. If $L_K(E)$ is a unital algebra, then it can be decomposed as a finite direct sum of indecomposable Leavitt path algebras, which are uniquely determined up to isomorphism. That is, $L_K(E)\cong\bigoplus_{i=1}^n L_K(E_i)$, where each $L_K(E_i)$ is indecomposable and if  $L_K(E)\cong \bigoplus_{i=1}^m L_K(F_i)$, then $n=m$ and $L_K(E_i)\cong L_K(F_{\sigma(i)})$ for some permutation $\sigma$.
\end{corollary}

\begin{proof}
First, if $L_K(E)$ is indecomposable then there is nothing to prove. Otherwise, if $L_K(E)$ is decomposable, following the proof of Proposition \ref{secondimplication} then there exist nontrivial hereditary and saturated subsets $X$ and $Y$ such that $L_K(E)=I_{(X,B_X)}\oplus I_{(Y,B_Y)}$. Then, using the construction of the quotient graphs and isomorphisms detailed in \S 3, we have $$I_{(X,B_X)}\cong L_K(E)/I_{(Y,B_Y)}\cong L_K(E\setminus (Y,B_Y)),$$ where the quotient graphs now reduce to: \begin{align*}
(E\backslash (Y,B_Y))^{0}& =(E^{0}\backslash Y)\cup \{v^{\prime }:v\in
B_{Y}\backslash B_Y\}=E^{0}\backslash Y; \\
(E\backslash (Y,B_Y))^{1}& =\{e\in E^{1}:r(e)\notin Y\}\cup \{e^{\prime }:e\in
E^{1},r(e)\in B_{Y}\backslash B_Y\}=r^{-1}(E^1\setminus Y).
\end{align*} Similarly we can say that $I_{(Y,B_Y)}\cong L_K(E\setminus (X,B_X))$ so that $L_K(E)\cong L_K(E\setminus (X,B_X))\oplus L_K(E\setminus (Y,B_Y))$.

If $L_K(E)$ is unital then $E^0$ is a finite set. Since $X,Y\neq \emptyset$ then  $|(E\backslash (X,B_X))^{0}|=|E^{0}\backslash X|<|E^0|$, and also $|(E\backslash (Y,B_Y))^{0}|<|E^0|$. This means that we can apply the same procedure to $L_K(E\setminus (X,B_X))$ and $L_K(E\setminus (Y,B_Y))$, and after at most $|E^0|$ steps we arrive to an indecomposable decomposition. Note that we are using that the ideals in which we are decomposing are isomorphic to Leavitt path algebras, and therefore have local units, which in turn allows us to say that being an ideal is a transitive relation, i.e., if $I$ is an ideal of $J$ and $J$ is an ideal of $L_K(E)$, then $I$ is an ideal of $L_K(E)$ (see for instance \cite[Lemma 4.14]{T2}). This shows that the decomposition exists.

For the uniqueness, let $\varphi: \bigoplus_{i=1}^n L_K(E_i) \to \bigoplus_{i=1}^m L_K(F_i)$ be an algebra isomorphism. Consider $I_i=\varphi(L_K(E_i))$ which, since $\varphi$ is an isomorphism, are ideals in $\bigoplus_{i=1}^m L_K(F_i)$ such that $\bigoplus_{i=1}^n I_i= \bigoplus_{i=1}^m L_K(F_i)$. Now apply \cite[Lemma (3.8)]{L} to get that $n=m$ and $I_i= L_K(F_{\sigma(i)})$ for some permutation $\sigma$. Now since $I_i=\varphi(L_K(E_i))\cong L_K(E_i)$ we get the desired conclusion.
\end{proof}

We can now give the characterization of the indecomposable Leavitt path algebras in the simpler case that the graph $E$ is finite and row-finite:

\begin{corollary}\label{rowfinitecase} Let $E$ be a finite graph. The Leavitt path algebra $L_K(E)$ is decomposable if and only if there exist nontrivial hereditary and saturated subsets $X$ and $Y$ with $X\cap Y=\emptyset$ such that the following conditions hold:
\begin{enumerate}
\item For every $v\in E^0$ there exists a path joining $v$ with $X\cup Y$.
\item The graph $E^0\setminus (X\cup Y)$ is acyclic.
\end{enumerate}
\end{corollary}
\begin{proof}
Since $E$ does not contain infinite emitters by hypothesis, an application of Theorem \ref{theresult} shows that condition (i) in the Theorem is then a tautology. On the other hand, the existence of a path $\alpha=\alpha_1\dots \alpha_k$ in Theorem \ref{theresult} corresponds to (1) in the statement of the Corollary. The fact that there are finitely many $XY$-compatible paths is related to item (2) of the statement as follows:

Suppose that there exist infinitely many $XY$-compatible paths $\alpha=\alpha_1\dots \alpha_k$ starting at $v$. Since $E$ is finite, there are finitely many vertices to choose from for the $s(\alpha_i)$. But in addition to this, since $E$ is row-finite, this necessarily implies that $E^1$ is finite as well. Thus, the only way we can get infinitely many $XY$-compatible paths is through the existence of a cycle $c$, having at least one vertex outside $X\cup Y$. Now, since $X\cup Y$ is hereditary, actually all the vertices $c^0$ must lie outside $X\cup Y$, which contradicts (2). The converse can be proved using similar ideas.
\end{proof}


The result in Corollary \ref{decomposition} can be generalized when the Leavitt path algebra satisfies a two-sided chain condition, concretely when it is either two-sided noetherian (i.e., it satisfies the a.c.c. in the set of two-sided ideals), or it is two-sided artinian (i.e., it satisfies d.c.c. on the set of two-sided ideals). These two conditions were characterized in \cite[Theorem 3.6]{ABCR} and \cite[Theorem 3.9]{ABCR} respectively for Leavitt path algebras of arbitrary graphs.

\begin{corollary}\label{twosidednoetherian} Let $E$ be an arbitrary graph. If $L_K(E)$ is either two-sided noetherian or two-sided artinian, then it can be decomposed as a finite direct sum of indecomposable Leavitt path algebras, which are uniquely determined up to isomorphism. That is, $L_K(E)\cong\bigoplus_{i=1}^n L_K(E_i)$, where each $L_K(E_i)$ is indecomposable and if  $L_K(E)\cong \bigoplus_{i=1}^m L_K(F_i)$, then $n=m$ and $L_K(E_i)\cong L_K(F_{\sigma(i)})$ for some permutation $\sigma$.
\end{corollary}
\begin{proof}
First we prove that the decomposition exists. If $L_K(E)$ is indecomposable there is nothing to prove, otherwise following the proof of Corollary \ref{decomposition} there exist (nontrivial) hereditary and saturated subsets $X$ and $Y$ of $E$ such that $L_K(E)=I_1\oplus I_2$ where $I_1=L_K(E\setminus (X,B_X))$ and $I_2=L_K(E\setminus (Y,B_Y))$. If both $I_1$ and $I_2$ are indecomposable, we are finished. Otherwise, suppose for instance that $I_2=L_K(E\setminus (Y,B_Y))$ is decomposable and again find nontrivial ideals $I_3$ and $I_4$, isomorphic to appropriate Leavitt algebras, such that $I_2=I_3\oplus I_4$, and again an application of \cite[Lemma 4.14]{T2} gives that $I_3$ and $I_4$ are ideals of $L_K(E)$ as well, so that we have a decomposition $L_K(E)\cong I_1\oplus I_3 \oplus I_4$. If this process were to continue, we would find a nontrivial decomposition for $I_4$ say, as $I_4=I_5\oplus I_6$, etc.

Thus, in the case that $L_K(E)$ was two-sided noetherian we would obtain an infinite ascending chain $I_1\subset I_1\oplus I_3 \subset I_1\oplus I_3\oplus I_5\subset ...$, which would give a contradiction, whereas in the situation that $L_K(E)$ was two-sided artinian, we would use the infinite descending chain $L_K(E)\supset I_2 \supset I_4 \supset ...$ to reach to the desired contradiction. In either case the indecomposable decomposition must therefore exist.

To see the uniqueness we use the argument given in Corollary \ref{decomposition}, where in order to make use of \cite[Lemma (3.8)]{L} we must make sure that if $R=B_1\oplus \dots \oplus B_n$ where $B_i$ is an ideal of $R$ (and $R$ might not be unital but has local units as does $L_K(E)$), then each ideal $I$ of $R$ has the form $I=I_1\oplus \dots \oplus I_n$, where $I_i$ is an ideal of the ring $B_i$.

Indeed, define $I_i=I\cap B_i$ and we claim that $I\subseteq I_1\oplus \dots \oplus I_n$. Let $a\in I$ and find $e=e^2\in R$ a local unit for $a$ such that $a=ae$. Write $e=e_1+\dots+e_n$ with $e_i=e_i^2\in B_i$. Then $a=ae=ae_1+\dots+ae_n\in I_1\oplus \dots \oplus I_n$. The reverse containment is obvious. This finishes the proof.

\end{proof}

We conclude the paper with several examples which illustrate our results.

\begin{examples} {\rm  (i) There exist indecomposable Leavitt path algebras which are not simple: Let $E_\infty$ denote the infinite edges graph given by
\[\begin{array}{cc} E_\infty \ = &\xymatrix{{\bullet}^w  \ar[r]^{(\infty)} & {\bullet}^x}\end{array}\] where the label $(\infty)$ denotes an infinite set of edges. By \cite[Examples 3.2]{AA3} this algebra is not simple. However, it is indecomposable by Theorem \ref{theresult} because the only nontrivial hereditary and saturated set is $\{x\}$.

\medskip

(ii) The row-finite and arbitrary graph cases are different, in the following sense: consider the graph $E$ given by $$
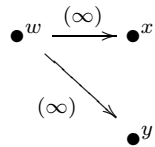
\xymatrix{{\bullet}^w  \ar[r]^{(\infty)} \ar[dr]_{(\infty)}& {\bullet}^x \\   &  {\bullet}^y}$$ Then the only nontrivial hereditary and saturated subsets are $X=\{x\}$ and $Y=\{y\}$, so if $L_K(E)$ was decomposable, then $(X,Y)$ would be the only possible pair to be considered in Theorem \ref{theresult}. Now since $w\not\in B_X\cup B_Y$, then all the edges in the graph are $XY$-compatible edges so that there exist infinitely many $XY$-compatible paths starting at $w=E^0\setminus (X\cup Y)$, a contradiction. Thus $L_K(E)$ is indecomposable.

However, the two conditions of Corollary \ref{rowfinitecase} are clearly satisfied, with $E^0$ finite, which shows that the hypothesis of the row-finiteness in Corollary \ref{rowfinitecase} cannot be removed (otherwise it would yield that $L_K(E)$ is decomposable, which we know is false).

\medskip

(iii) The condition of the finiteness of the graph in Corollary \ref{rowfinitecase} cannot be removed either. To show this, consider the infinite graph $F$ given by $$
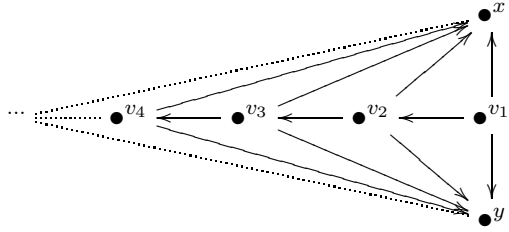
\xymatrix{  & & & & {\bullet}^x \\ ^{\dots} \ar@{.}[r] \ar@{.}[urrrr]\ar@{.}[drrrr]& {\bullet}^{v_4}  \ar[urrr] \ar[drrr]  &  {\bullet}^{v_3} \ar[l] \ar[urr] \ar[drr] & {\bullet}^{v_2} \ar[l] \ar[ur] \ar[dr] & {\bullet}^{v_1} \ar[l] \ar[u] \ar[d]  \\   & & & & {\bullet}^y}$$ Then the only nontrivial hereditary and saturated sets are $X=\{x\}$ and $Y=\{y\}$ (if any $v_i$ was contained in some other hereditary and saturated subset $H$, by hereditariness we have $v_j\in H$ for all $j> i$, and also $x,y\in H$. At this point apply saturation to get $v_{i-1}\in H$, and go backwards in the graph to get $v_{i-2}\in H$, etc.)

The graph $F$ does not have infinite emitters, and since $x$ and $y$ are sinks, then clearly every edge in the graph is a $XY$-compatible edge, so that there exist infinitely many $XY$-compatible paths starting at every $v_k$. Then Theorem \ref{theresult} gives that $L_K(F)$ is indecomposable.

Now suppose that Corollary \ref{rowfinitecase} could be applied (the only hypothesis which is not satisfied is the finiteness of $F^0$). In this scenario, since conditions (1) and (2) in Corollary \ref{rowfinitecase} are clearly satisfied for $(X,Y)$, it would give that $L_K(F)$ is decomposable, a contradiction.

\medskip

(iv) The following example illustrates the subtleties which appear with infinite emitters. Consider the following three graphs $E_1$, $E_2$ and $E_3$ given by
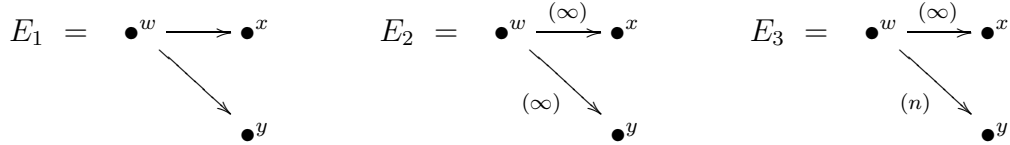
\[\begin{array}{cccccc} E_1 \ = & \xymatrix{{\bullet}^w  \ar[r] \ar[dr] & {\bullet}^x \\   &  {\bullet}^y} &
\phantom{space} E_2 \ = & \xymatrix{{\bullet}^w  \ar[r]^{(\infty)} \ar[dr]_{(\infty)}& {\bullet}^x \\   &  {\bullet}^y}
&
\phantom{space} E_3 \ = & \xymatrix{{\bullet}^w  \ar[r]^{(\infty)} \ar[dr]_{(n)}& {\bullet}^x \\   &  {\bullet}^y}\end{array}\] where $(n)$ denotes that there are $n\in {\mathbb N}$ parallel edges from $w$ to $y$. From \cite[Proposition 3.5]{AAS} we know that $L_K(E_1)\cong {\mathbb M}_2(K)\oplus {\mathbb M}_2(K)$ and therefore it is decomposable.

However, if we have infinitely many edges, as we have seen in (ii), we get that the graph $E_2$ gives an indecomposable Leavitt path algebra $L_K(E_2)$. Note that this happens despite having a ``bifurcation" at $w$, which is what essentially gives the direct sum decomposition in \cite[Proposition 3.5]{AAS}.

Moreover, by further changing the graph to $E_3$ we again get a decomposable Leavitt path algebra $L_K(E_3)$: the reason is that $X:=\{x\}$ and $Y:=\{y\}$ meet the hypotheses of Theorem \ref{theresult}. Indeed, $w\in B_X\setminus B_Y$ and also none of the edges in $r^{-1}(x)$ can be $XY$-compatible edges starting at $w$; but at the same time, all the edges in $r^{-1}(y)$ are $XY$-compatible, so that there exist finitely many $XY$-compatible paths starting at $w$.

\medskip

(v) Furthermore, we can get the indecomposable decomposition of the graph $E_3$ above. Following (iv), the pair $X=\{x\}$ and $Y=\{y\}$ satisfy the hypotheses of Theorem \ref{theresult}, so that we can follow the proof of Corollary \ref{decomposition} and we get that $L_K(E_3)\cong L_K(E_3\setminus (X,B_X))\oplus L_K(E_3\setminus (Y,B_Y))$. Now the quotient graph constructions give that \[\begin{array}{cccc} E_3\setminus (X,B_X) \ = & \xymatrix{{\bullet}^w  \ar[r]^{(n)} & {\bullet}^y} &
\phantom{space} E_3\setminus (Y,B_Y) \ = & \xymatrix{{\bullet}^w  \ar[r]^{(\infty)} & {\bullet}^x}\end{array}\] so that an application of \cite[Proposition 3.5]{AAS} and the previous item (i), respectively, give that $L_K(E_3)\cong {\mathbb M}_n(K)\oplus L_K(E_\infty)$ is the desired indecomposable decomposition of $L_K(E_3)$.

Note however that for the graph $E_2$ in (iv), we have that $L_K(E_2)\not\cong L_K(E_\infty)
\oplus L_K(E_\infty)$, contrary to intuition.
}

\end{examples}


\section*{acknowledgements}

The authors are thankful to the referee for his/her careful reading of the paper and for some helpful
comments and suggestions that improved the presentation of the paper. Also we thank professor P. Ara, E. Pardo and M. Tomforde for valuable correspondence and input. The first author was partially supported by the Spanish MEC and Fondos FEDER through project MTM2010-15223, and by the Junta de Andaluc\'{\i}a and Fondos FEDER, jointly, through projects FQM-336 and FQM-2467. Part of this work was carried out during a visit of the first author to the University of Isfahan, Iran. The first author thanks this host institution for its warm hospitality and support. The research of the second author was in part supported by a grant from IPM(No. 91160042). The second author would like to thank the Banach Algebra Center of Excellence for Mathematics, University of Isfahan.


\begin{thebibliography}{10}

\bibitem{AA1} \textsc{G. Abrams, G. Aranda Pino}, The Leavitt path algebra of a graph, \emph{J. Algebra} \textbf{293 (2)} (2005), 319--334.

\bibitem{AA2} \textsc{G. Abrams, G. Aranda Pino}, Purely infinite simple Leavitt path algebras, \emph{J. Pure Appl. Algebra} \textbf{207 (3)}, (2006), 553--563.

\bibitem{AA3} \textsc{G. Abrams, G. Aranda Pino}, The Leavitt path algebras of arbitrary graphs, \emph{Houston J. Math.}, \textbf{34 (2)}, (2008), 423--442.

\bibitem{AAPS} \textsc{G. Abrams, G. Aranda Pino, F. Perera, M. Siles Molina}, Chain conditions for Leavitt path algebras.  \emph{Forum Math.} {\bf 22} (2010), 95--114.

\bibitem{AAS} \textsc{G. Abrams, G. Aranda Pino, M. Siles Molina}, Finite-dimensional Leavitt path algebras, \emph{J. Pure Appl. Algebra},Ê\textbf{ 209 (3)}, (2007), 753--762.

\bibitem{AAS2} \textsc{G. Abrams, G. Aranda Pino, M. Siles Molina}, Locally finite Leavitt path algebras, \emph{Israel J. Math.}  \textbf{165} (2008), 329--348.

\bibitem{ABCR} \textsc{G. Abrams, J. P. Bell, P. Colak, K. M. Rangaswamy}, Two-sided chain conditions in Leavitt path algebras over arbitrary graphs, \emph{J. Algebra Appl.}  \textbf{11 (3)} (2012), 1250044, 23p.

\bibitem{Aprivate} \textsc{P. Ara}, \emph{Private Communication} (2012).

\bibitem{AMP} \textsc{P. Ara, M.A. Moreno, E. Pardo}, Nonstable $K$-theory for graph algebras, \emph{Algebr. Represent. Theory}, \textbf{10} (2007), 157--178.

\bibitem{AG} \textsc{P. Ara, K. Goodearl}, Leavitt path algebras of separated graphs, \emph{J. Reine Angew.} (to appear). Arxiv:1005.1900

\bibitem{AC} \textsc{G. Aranda Pino, K. Crow}, The center of a Leavitt path algebras, \emph{Rev. Mat Iberoam.} \textbf{27 (2)}, (2011), 621--644.

\bibitem{APS} \textsc{G. Aranda Pino, E. Pardo, M. Siles Molina}. Exchange Leavitt path algebras and stable rank. \emph{J. Algebra} \textbf{305(2)} (2006), 912--936.

\bibitem{ARS2} \textsc{G. Aranda Pino, K. M. Rangaswamy, M. Siles Molina}, Weakly regular and self-injective Leavitt path algebras, \emph{Algebr. Represent. Theory}, \textbf{14}, (2011), 751--777.

\bibitem{ARV} \textsc{G. Aranda Pino. K. L. Rangaswamy, L. Va\v s}, $^\ast$-regular Leavitt path algebra of arbitrary graphs, \emph{Acta Math. Sin.}, \textbf{28 (5)}, (2012), 957--968.

\bibitem{ASS} \textsc{I. Assem, D. Simson, A. Skowro\'{n}ski}, \emph{Elements of the representation theory of associative algebras}, London Math. Soc. Student Texts \textbf{65}, Cambridge University Press, 2006.

\bibitem{BPRS} \textsc{T. Bates, D. Pask, I. Raeburn, W. Szyma\'nski}, The $C^*$-algebras of row-finite
graphs, \emph{New York J. Math.} \textbf{6} (2000), 307--324.

\bibitem{Cuntz} \textsc{J. Cuntz}, Simple $C^\ast$-algebras generated by isometries, \emph{Comm. Math. Phys. } \textbf{57} (1977), 173--185.

\bibitem{DT} \textsc{D. Drinen, M. Tomforde}, The $C^*$-algebras of arbitrary graphs, \emph{Rocky Mountain J. Math.} \textbf{35 (1)} (2005), 105--135.

\bibitem{G2} \textsc{K. R. Goodearl}, Leavitt path algebras and direct limits, \emph{Contemp. Math.} \textbf{480} (2009), 165--187.

\bibitem{H} \textsc{J. H. Hong}, Decomposability of graph $C^*$-algebras, \emph{Proc. Amer. Math. Soc.} \textbf{133 (1)} (2004), 115--126.

\bibitem{KPR} \textsc{A. Kumjian, D. Pask, I. Raeburn}, Cuntz-Krieger algebras of directed graphs,
\emph{Pacific J. Math.} \textbf{184} (1998), 161--174.

\bibitem{L} \textsc{T. Y. Lam}, \emph{A first course in noncommutative rings}, Springer-Verlag New York Inc., 1991.

\bibitem{Leavitt} \textsc{W. G. Leavitt}, Modules without invariant basis number, \emph{Proc. Amer. Math.
Soc.} \textbf{8} (1957), 322--328.

\bibitem{Pprivate} \textsc{E. Pardo}, \emph{Private Communication} (2012).

\bibitem{Raeburn} \textsc{I. Raeburn}. \emph{Graph algebras}. CBMS Regional Conference Series in Mathematics \textbf{103}, Amer. Math. Soc., Providence (2005).

\bibitem{Ruiz} \textsc{E. Ruiz, M. Tomforde}, Ideals in graph algebras, \emph{Algebr. Represent. Theory},
http://arxiv.org/abs/1205.1247.

\bibitem{S} \textsc{M. Siles Molina}, Algebras of quotients of path algebras,  \emph{J. Algebra} \textbf{319 (12)} (2008), 329--348.

\bibitem{T} \textsc{M. Tomforde}, Uniqueness theorems and ideal structure
for Leavitt path algebras, \emph{J. Algebra} \textbf{318} (2007), 270--299.

\bibitem{T2} \textsc{M. Tomforde}, Leavitt path algebras with coefficients in a commutative ring, \emph{J. Algebra} \textbf{215} (2011), 471--484.


\bibitem{Tprivate} \textsc{M. Tomforde}, \emph{Private Communication} (2012).

\end{thebibliography}
\end{document}